\documentclass[12pt]{article}

\usepackage{mathpazo}
\usepackage{times}

\usepackage{amsthm}
\usepackage{amsmath}
\usepackage{amsfonts}
\usepackage{graphicx}
\usepackage{latexsym}
\usepackage{amssymb}

\newcommand{\deff}{\mbox{$\stackrel{\rm def}{=}$}}

\newcommand{\field}[1]{\mathbb{#1}}

\newcommand{\Z}{\field{Z}}

\newcommand{\F}{\field{F}}

\newcommand{\dS}{\field{S}}

\newcommand{\cG}{{\cal G}}

\newtheorem{theorem}{Theorem}
\newtheorem{lemma}{Lemma}

\newtheorem{cor}{Corollary}
\newtheorem{conj}{Conjecture}

\makeatletter

\gdef\@punct{.\ \ }  
\def\@sect#1#2#3#4#5#6[#7]#8{%
  \ifnum #2>\c@secnumdepth
     \def\@svsec{}
  \else
     \refstepcounter{#1}\edef\@svsec{%
     \ifnum #2>0{{\csname the#1\endcsname}}.\fi%
    \hskip .5em}
  \fi
  \@tempskipa #5\relax
  \ifdim \@tempskipa>\z@
     \begingroup #6\relax
       \@hangfrom{\hskip #3\relax\@svsec}{\interlinepenalty \@M #8\par}
     \endgroup
     \csname #1mark\endcsname{#7}
     \addcontentsline{toc}{#1}{\ifnum #2>\c@secnumdepth\else
          \protect\numberline{\csname the#1\endcsname}\fi#7}
  \else
     \def\@svsechd{#6\hskip #3\@svsec #8\@punct\csname #1mark\endcsname{#7}
     \addcontentsline{toc}{#1}{\ifnum #2>\c@secnumdepth \else
          \protect\numberline{\csname the#1\endcsname}\fi#7}}
  \fi
  \@xsect{#5}}

\def\@ssect#1#2#3#4#5{\@tempskipa #3\relax
  \ifdim \@tempskipa>\z@
    \begingroup #4\@hangfrom{\hskip #1}{\interlinepenalty \@M #5\par}\endgroup
  \else \def\@svsechd{#4\hskip #1\relax #5\@punct}\fi
  \@xsect{#3}}

\makeatother

\begin{document}

\bibliographystyle{plain}

\title{\huge\bf $q$-Analogs of Steiner Systems\\[1.00ex]}

\author{
{\Large\sc Tuvi Etzion}\thanks{Department of Computer Science, Technion,
Haifa 32000, Israel, e-mail: {\tt etzion@cs.technion.ac.il}.} \and
{\Large\sc Alexander Vardy}\thanks{
University of California San Diego, La Jolla, CA 92093, USA,
e-mail: {\tt avardy@ucsd.edu}.}}

\date{$\,$\\ November 9, 2012}

\maketitle

\thispagestyle{empty}

\begin{abstract}
\noindent
A {Steiner structure} $\dS = \dS_q[t,k,n]$ is a set 
of $k$-dimensional subspaces of~$\F_q^n$ such that
each $t$-dim\-ensional subspace of $\F_q^n$ is contained in exactly one
subspace of $\dS$. 
Steiner structures~are the~$q$-analogs of Steiner systems;
they are presently known to exist only for $t = 1$, $t=k$, 
and\linebreak for $k = n$.
The existence of nontrivial $q$-analogs of Steiner systems
has occupied mathematicians for over three decades. In fact,
it was conjectured that they do not exist.
In this paper,~we show that 
nontrivial Steiner structures do exist.
First, we describe a general method which may be used
to produce Steiner structures. The method uses two
mappings in a finite field: the Frobenius map and
the cyclic shift map. These maps are applied to codes
in the Grassmannian, in order to form an automorphism
group of the Steiner structure. Using this method,
assisted by an exact-cover computer search, we explicitly
generate a Steiner structure $\dS_2[2,3,13]$. We conjecture
that many other Steiner structures, with different parameters,
exist.
\end{abstract}

\vspace{0.5cm}


\footnotetext[1]{Research supported in part by the Israeli
Science Foundation (ISF), Jerusalem, Israel, under
Grant 10/12}

\newpage
\section{Introduction}
\label{sec:introduction}

The Grassmannian $\cG_q(n,k)$ is the set of all
$k$-dimensional subspaces of an $n$-dimensional
subspace over the finite field $\F_q$. A code
in $\cG_q(n,k)$ is a subset
of $\cG_q(n,k)$. There has been lot of interest in
these codes in the last five years due to their
application in network coding~\cite{KoKs}.
Our motivation for this work also came from this application in
network coding.

This paper is devoted to the existence of nontrivial Steiner
structures, known also as $q$-analog of Steiner systems.
A \emph{Steiner structure} $\dS_q[t,k,n]$ is a set $\dS$
of $k$-dimensional subspaces of~$\F_q^n$ such that
each $t$-dimensional subspace of $\F_q^n$ is contained in exactly one
subspace of $\dS$. Steiner structures were considered in many
papers~\cite{AAK,EtVa11,ScEt02,Tho1,Tho2}, where they have other names as well.
An $\dS_q[t,k,n]$ can be readily constructed for $t=k$
and for $k=n$. If $t=1$ these structures are called
$k$-\emph{spreads} and they are known to exist if and only
if $k$ divides~$n$. These structures are also considered to be
trivial. The first nontrivial case is a Steiner structure
$\dS_2[2,3,7]$. The possible existence of this structure was considered
by several authors, and some conjectured~\cite{Met99} that it doesn't exist
and that generally nontrivial
Steiner structures do not exist.

The main result of this paper is a description of a method to search
for a structured Steiner structure $\dS_p[2,k,n]$. The search was
successful in finding a Steiner structure $\dS_2[2,3,13]$.

The rest of this paper is organized as follows.
In Section~\ref{sec:map} we define two types of mappings
and state some of their properties. In Section~\ref{sec:Steiner}
we use the two types of mappings for an attempt to construct
nontrivial Steiner structures $\dS_p [2,k,n]$, where $p$ and $n$ are prime
integers.
In Section~\ref{sec:S2_3_13} we construct a Steiner structure $\dS_2[2,3,13]$.
In Section~\ref{sec:more} we discuss the possible existence of
more Steiner systems with the same parameters and with different parameters.
In Section~\ref{sec:related} we discuss some more combinatorial designs which are
related to the constructed Steiner structures.
Conclusion and problems for further research are given
in Section~\ref{sec:conclude}.

\section{Mappings in the Grassmannian}
\label{sec:map}

Let $\F_{p^n}$ be the finite field with $p^n$ elements,
where $p$ and $n$ are primes,
and let $\alpha$ be a primitive element of $\F_{p^n}$.

The \emph{Frobenius map} $\Upsilon_\ell$,
$0 \leq \ell \leq n-1$, $\Upsilon_\ell : \F_{p^n} \setminus \{ 0 \}
\longrightarrow \F_{p^n} \setminus \{ 0 \}$
is defined by $\Upsilon_\ell (x) \deff x^{p^\ell}$ for each
$x \in \F_{p^n} \setminus \{ 0 \}$.

The \emph{cyclic shift map} $\Phi_j$, $0 \leq j \leq p^n-2$,
$\Phi_j : \F_{p^n} \setminus \{ 0 \}
\longrightarrow \F_{p^n} \setminus \{ 0 \}$ is defined by
$\Phi_j (\alpha^i) \deff \alpha^{i+j}$, for each $0 \leq i \leq p^n-2$.

When we say the Frobenius map or the cyclic shift map, without specifying
a particular map, we will
mean $\Upsilon_1$ or $\Phi_1$, respectively.

The two types of mappings $\Upsilon_\ell$ and $\Phi_j$ can be applied
on a subset or a subspace, by applying the map on each
element of the subset or subspace, respectively. Formally,
given two integers $0 \leq \ell \leq n-1$ and $0 \leq j \leq p^n-2$,
$$
\Upsilon_\ell \{ x_1 , x_2 , \ldots , x_r \} \deff
\{ \Upsilon_\ell (x_1), \Upsilon_\ell (x_2),\ldots,\Upsilon_\ell(x_r) \}~,
$$
$$
\Phi_j \{ x_1 , x_2 , \ldots , x_r \} \deff \{ \Phi_j (x_1), \Phi_j
(x_2),\ldots,\Phi_j (x_r) \}.
$$

\begin{lemma}
\label{lem:inverse}
The mappings $\Upsilon_\ell$ and $\Phi_j$ are invertible.
\end{lemma}
\begin{proof}
Clearly, $\Upsilon_\ell^{-1} = \Upsilon_{n-\ell}$ and
$\Phi_j^{-1} = \Phi_{2^n-1-j}$.
\end{proof}

For a given integer $s \in \Z_{p^n-1}$,
the cyclotomic coset $C_s$ is defined by
$$
C_s \deff \{ s \cdot p^i ~:~ 0 \leq i \leq n-1 \}~.
$$
The smallest element in a cyclotomic coset is called
the \emph{coset representative}. Let $\rho (s)$ denote the
coset representative of $C_s$, i.e. if $r$ is the coset
representative for the coset of $s$, $C_s$,
then $r=\rho (s)$. The following lemmas are well-known and can be easily
verified.

\begin{lemma}
The size of a cyclotomic coset is either $n$ or \emph{one}. There are
exactly $\frac{p^n-p}{n}$ different cyclotomic cosets of size $n$.
\end{lemma}

\begin{lemma}
When applied on $\F_{p^n} \setminus \{ 0 \}$ the
Frobenius mappings forms an equivalence relation on
$\F_{p^n} \setminus \{ 0 \}$,
where  an equivalence class contains the powers of $\alpha$
which are exactly the elements of one cyclotomic cosets of
$\Z_{p^n-1}$.
\end{lemma}

\begin{lemma}
\label{lem:isomorphism}
The finite field $\F_{p^m}$ and the vector space $\F_p^n$
are isomorphic.
\end{lemma}

In view of Lemma~\ref{lem:isomorphism} we can apply the Frobenius mappings and
the cyclic shifts mappings on $\F_p^m$ exactly as they are applies
on $\F_{p^m}$. If $h: \F_{p^m} \longrightarrow \F_p^n$ is the
isomorphism from $\F_{p^m}$ to $\F_p^m$ such that $y=h(x)$ for
$y \in \F_p^m$ and $x \in \F_{p^m}$ then
$$
\Upsilon_\ell (y) \deff h(\Upsilon_\ell (x)) ~~ \text{and} ~~ \Phi_j (y) \deff h(\Phi_j (x)) ~,
$$
for every $0 \leq \ell \leq n-1$ and $0 \leq j \leq p^n-2$.

\begin{lemma}
\label{lem:equiv}
Let $X$ and $Y$ be two $k$-dimensional subspaces of $\F_p^n$
such that there exist two integers, $\ell_1$, $0 \leq \ell_1 \leq n-1$,
and $j_1$, $0 \leq j_1 \leq p^n-1$, such that $Y= \Phi_{j_1} (\Upsilon_{\ell_1} (X))$.
Then there exist two integers, $\ell_2$, $0 \leq \ell_2 \leq n-1$,
and $j_2$, $0 \leq j_2 \leq p^n-1$, such that $Y= \Upsilon_{\ell_2} ( \Phi_{j_2} (X))$.
\end{lemma}
\begin{proof}
Let $X$ and $Y$ be two $k$-dimensional subspaces of $\F_p^n$
such that there exist two integers, $\ell_1$, $0 \leq \ell_1 \leq n-1$,
and $j_1$, $0 \leq j_1 \leq p^n-1$, such that $Y= \Phi_{j_1} (\Upsilon_{\ell_1} (X))$.
Then
$$
Y = \alpha^j X^{p^{\ell_1}} \Longrightarrow Y^{p^{n-\ell_1}} = (\alpha^{j_1} X^{p^{\ell_1}})^{p^{n-\ell_1}} \Longrightarrow \alpha^{p^n-1-j_1p^{n-\ell_1}} Y^{p^{n-\ell_1}}=X~,
$$
and hence $X= \Phi_{{p^n-1-j_1p^{n-\ell_1}}} (\Upsilon_{p^{n-\ell_1}}(Y))$. Thus,
$Y=  \Upsilon_{\ell_1} (\Phi_{j_1p^{n-\ell_1}}(X))$.
\end{proof}
Similar manipulations as the ones in Lemma~\ref{lem:equiv} lead to the following consequence.
\begin{cor}
The combination of the Frobenius map and the cyclic shift mappings
induces an equivalence relation of the set of all $k$-dimensional subspaces
of $\F_p^n$.
\end{cor}

\section{Steiner Structures $\dS_p[2,k,n]$}
\label{sec:Steiner}

We suggest to construct a set $\dS$ of $k$-dimensional
subspaces of $\F_p^n$, $n$ prime,
which has the cyclic shift map
and the Frobenius map as generators for its automorphism groups.
For simplicity we will consider the case where $p=2$,
but the method is easily generalized for larger primes.
The nonzero elements of the field will be represented as one cycle
for this construction. Given a $k$-dimensional subspace
$$\{ {\bf 0} , \alpha^{i_1}, \alpha^{i_2} , \ldots , \alpha^{i_{2^k-1}} \}$$
in $\dS$, we require that for each $0 \leq \ell \leq n-1$
and $0 \leq j \leq 2^n-2$,
$$\{ {\bf 0} , \alpha^{i_1 2^\ell +j}, \alpha^{i_2 2^\ell +j} , \ldots , \alpha^{i_{2^k-1} 2^\ell +j} \}$$
will be also a $k$-dimensional subspace of $\dS$. In other words, $X \in \dS$ if and only if
${\Phi_j (  \Upsilon_\ell (X)) \in \dS}$, for every  $0 \leq \ell \leq n-1$ and $0 \leq j \leq p^n-2$.

For a given $k$-dimensional subspace
$$X=\{ {\bf 0} , \alpha^{i_1}, \alpha^{i_2} , \ldots , \alpha^{i_{2^k-1}} \}$$
of $\F_2^n$, let the \emph{difference set} of $X$, $\Delta (X)$,
be the set of integers defined by
$$
\Delta (X) \deff \{ i_r - i_s ~:~  1 \leq r,~s \leq 2^k-1,~ r \neq s \}~.
$$

\begin{lemma}
If $X$ is a $k$-dimensional subspace of $\F_2^n$, where $\gcd (2^k-1,2^n-1)=1$,
then the cyclic shifts of $X$ form $2^n-1$ distinct $k$-dimensional subspaces.
\end{lemma}

A $k$-dimensional subspace $X$ of $\F_2^n$ will be called \emph{complete}
if $| \Delta (X)| =(2^k-1)(2^k-2)$. Two complete $k$-dimensional subspaces $X,~Y$ of $\F_2^n$
will be called \emph{disjoint complete} if $\Delta (X) \cap \Delta (Y) = \varnothing$.
Each $k$-dimensional subspace $X$ of $\F_2^n$ contains
$\frac{(2^k-1)(2^{k-1}-1)}{3}$ two-dimensional
subspaces of $\F_2^n$. If $| \Delta (X)| =(2^k-1)(2^k-2)$ then no two of these
$\frac{(2^k-1)(2^{k-1}-1)}{3}$ two-dimensional subspaces are cyclic shifts of each other.
Hence, we have

\begin{lemma}
If $X$ is a $k$-dimensional subspace of $\F_2^n$ and $| \Delta (X)| = \frac{(2^k-1)(2^{k-1}-1)}{3}$
then the cyclic shifts of $X$ form $2^n-1$ distinct $k$-dimensional subspaces.
The $\frac{(2^k-1)(2^{k-1}-1)}{3} \cdot (2^n-1)$ two-dimensional subspaces of these
$2^n-1$ distinct $k$-dimensional subspaces are all distinct.
\end{lemma}
The first consequence of the theory is
\begin{theorem}
\label{thm:cyclic}
If there exist $\frac{2^n-2}{(2^k-1)(2^k-2)}$ pairwise disjoint complete
$k$-dimensional subspaces then there exists a Steiner structure
$\dS_2[2,k,n]$.
\end{theorem}

A Steiner structure derived via Theorem~\ref{thm:cyclic} has an automorphism which
consists of a cycle of length $2^n-1$. Such a Steiner system will be called cyclic.
Related codes in the Grassmannian which have automorphism of size $2^n-1$
are called cyclic codes. Such codes were considered in~\cite{EV11,KoKu}.
The search for such large codes become time consuming and less efficient and
hence another tool to simplify the search (on the expense of a more
structured code) should be added.
Therefore, we try to add one more automorphism group into the structure
by adding the Frobenius mappings into the equation.

For a given $k$-dimensional subspace
$$X=\{ {\bf 0} , \alpha^{i_1}, \alpha^{i_2} , \ldots , \alpha^{i_{2^k-1}} \}$$
of $\F_2^n$ let the \emph{coset difference set} of $X$, $\rho(\Delta (X))$,
be the set of integers defined by
$$
\rho(\Delta (X)) \deff \{ \rho (i_r - i_s) ~:~  1 \leq r,~s \leq 2^k-1,~ r \neq s \}~.
$$

A $k$-dimensional subspace $X$ of $\F_2^n$ will be called \emph{coset complete}
if $| \rho(\Delta (X))| =(2^k-1)(2^k-2)$. Two coset complete $k$-dimensional subspaces $X,~Y$ of $\F_2^n$
will be called \emph{disjoint coset complete} if $\rho (\Delta (X)) \cap \rho (\Delta (Y)) = \varnothing$.

\begin{theorem}
\label{thm:pairwise}
If $n$ is a prime and
there exist $\frac{2^n-2}{(2^k-1)(2^k-2)n}$ pairwise disjoint coset complete
$k$-dimensional subspaces then there exists a Steiner structure
$\dS_2[2,k,n]$.
\end{theorem}

A search for pairwise disjoint
coset complete $k$-dimensional subspaces of $\F_2^n$ is done as follows.
$\{ {\bf 0}, \alpha^{i_1} , \alpha^{i_2}, \alpha^{i_3} \}$
is a two-dimensional subspace if and only if
$\alpha^{i_1} + \alpha^{i_2} + \alpha^{i_3}=0$. Also,
$\{ {\bf 0}, \alpha^{i_1} , \alpha^{i_2}, \alpha^{i_3} \}$
is a two-dimensional subspace if and only if
$\{ {\bf 0}, \alpha^{i_1+j} , \alpha^{i_2+j}, \alpha^{i_3+j} \}$
is a two-dimensional subspace for every integer $j$. Therefore,
$\rho (i_2-i_1)$, $\rho (i_1-i_2)$, $\rho (i_3-i_1)$, $\rho (i_1-i_3)$,
$\rho (i_3-i_2)$, $\rho (i_2-i_3)$, always appear together in
a coset difference set. It follows that we can partition the
$\frac{2^n-2}{n}$ cyclotomic cosets of size $n$,
into $\frac{2^n-2}{6n}$ groups and instead of $(2^k-1)(2^k-2)$ integers in a coset
difference set we should consider only $\frac{(2^k-1)(2^{k-1}-1)}{3}$ such integers
representing such a $k$-dimensional subspace.
We form a graph $G(V,E)$ as follows. The set $V$ of vertices
for $G$ are represented by $\frac{(2^k-1)(2^{k-1}-1)}{3}$-subsets
of the set of $\frac{2^n-2}{6n}$ elements
which represents the $\frac{2^n-2}{6n}$ groups of the cosets.
Such a $\frac{(2^k-1)(2^{k-1}-1)}{3}$-subset
$v$ represents a vertex if and only if there exists a $k$-dimensional
subspace $X$ of $\F_2^n$ whose coset difference set $\rho (\Delta(X)))$
is represented by $v$. For two vertices
$v_1,v_2 \in V$ represented by $\frac{(2^k-1)(2^{k-1}-1)}{3}$-subsets, there is an edge
$\{ v_1 , v_2 \} \in E$ if and only if $v_1 \cap v_2 = \varnothing$.
A clique with $m$ vertices in $G$ represents $m$ pairwise disjoint
coset complete $k$-dimensional subspaces. A clique with
$\frac{2^n-2}{(2^k-1)(2^k-1)n}$ vertices represents a Steiner structure
$\dS_2 [2,k,n]$.

\section{Steiner Structure $\dS_2[2,3,13]$}
\label{sec:S2_3_13}

A program which generates this graph for $n=13$ and $k=3$ was written.
For $n=13$ we have $\frac{2^{13}-2}{13} =630$ cyclotomic
cosets of size 13, and 105 groups.
A clique of size $\frac{2^n-2}{42n} =15$ in this graph represents a Steiner
structure $\dS_2[2,3,13]$. Unfortunately, it is not feasible
to check even a small fraction of the subsets with 15 vertices
of this graph and we found numerous number of cliques of size 14.
In the conference on "Trends in Coding Theory"
(Ascona, October 28, 2012 - November 2, 2012) Patric \"{O}sterg{\aa}rd
suggested to use the exact cover problem (to find 15 vertices with
disjoint integers for a total of 105 integers) to find a solution.
He pointed on a program in http://www.cs.helsinki.fi/u/pkaski/libexact/
based on the backtrack algorithm and the dancing links data structure of
Knuth~\cite{Knu00}. We added this algorithm into our search. We formed
GF($2^{13}$) with the primitive element
$\alpha$ which is a root of the primitive polynomial
$x^{13}+x^4+x^3+x+1=0$.
The following fifteen pairwise disjoint coset complete
3-dimensional subspaces given in Table~\ref{tab:solution}
form by Theorem~\ref{thm:pairwise} the required $\dS_2[2,3,13]$.

\begin{table}[ht]
\label{tab:solution}
\centering
\begin{large}
\begin{tabular}{|c|c|}
\hline $~$ &   representative
\tabularnewline \hline \hline
& \\
1 &  $\{ {\bf 0},\alpha^0, \alpha^1 , \alpha^{1249}, \alpha^{5040} , \alpha^{7258}, \alpha^{7978}, \alpha^{8105} \}$
\tabularnewline \hline
& \\
2 &  $\{ {\bf 0},\alpha^0, \alpha^7 , \alpha^{1857} , \alpha^{6681} , \alpha^{7259} , \alpha^{7381} , \alpha^{7908} \}$
\tabularnewline \hline
& \\
3 &  $\{ {\bf 0},\alpha^0, \alpha^9 , \alpha^{1144} , \alpha^{1945} , \alpha^{6771} , \alpha^{7714} , \alpha^{8102} \}$
\tabularnewline \hline
& \\
4 &  $\{ {\bf 0},\alpha^0, \alpha^{11}, \alpha^{209}, \alpha^{1941}, \alpha^{2926}, \alpha^{3565} , \alpha^{6579} \}$
\tabularnewline \hline
& \\
5 &  $\{ {\bf 0},\alpha^0, \alpha^{12} , \alpha^{2181}, \alpha^{2519} , \alpha^{3696} , \alpha^{6673}, \alpha^{6965} \}$
\tabularnewline \hline
& \\
6 &  $\{ {\bf 0},\alpha^0, \alpha^{13}, \alpha^{4821} , \alpha^{5178}, \alpha^{7823} , \alpha^{8052} , \alpha^{8110} \}$
\tabularnewline \hline
& \\
7 &  $\{ {\bf 0},\alpha^0, \alpha^{17} , \alpha^{291}, \alpha^{1199} , \alpha^{5132} , \alpha^{6266}, \alpha^{8057} \}$
\tabularnewline \hline
& \\
8 &  $\{ {\bf 0},\alpha^0, \alpha^{20} , \alpha^{1075} , \alpha^{3939}, \alpha^{3996} , \alpha^{4776}, \alpha^{7313} \}$
\tabularnewline \hline
& \\
9 &  $\{ {\bf 0},\alpha^0, \alpha^{21} , \alpha^{2900} , \alpha^{4226} , \alpha^{4915} , \alpha^{6087} , \alpha^{8008} \}$
\tabularnewline \hline
& \\
10 &  $\{ {\bf 0},\alpha^0, \alpha^{27}, \alpha^{1190} , \alpha^{3572} , \alpha^{4989}, \alpha^{5199} , \alpha^{6710} \}$
\tabularnewline \hline
& \\
11 &  $\{ {\bf 0},\alpha^0, \alpha^{30}, \alpha^{141} , \alpha^{682} , \alpha^{2024} , \alpha^{6256} , \alpha^{6406} \}$
\tabularnewline \hline
& \\
12 &  $\{ {\bf 0},\alpha^0, \alpha^{31} , \alpha^{814} , \alpha^{1161} , \alpha^{1243} , \alpha^{4434} , \alpha^{6254} \}$
\tabularnewline \hline
& \\
13 &  $\{ {\bf 0},\alpha^0, \alpha^{37} , \alpha^{258} , \alpha^{2093} , \alpha^{4703} , \alpha^{5396} , \alpha^{6469} \}$
\tabularnewline \hline
& \\
14 &  $\{ {\bf 0},\alpha^0, \alpha^{115} , \alpha^{949} , \alpha^{1272} , \alpha^{1580} , \alpha^{4539} , \alpha^{4873} \}$
\tabularnewline \hline
& \\
15 &  $\{ {\bf 0},\alpha^0, \alpha^{119} , \alpha^{490} , \alpha^{5941} , \alpha^{6670}, \alpha^{6812} , \alpha^{7312} \}$
\tabularnewline \hline

\end{tabular}
\end{large}
\caption{The fifteen pairwise disjoint coset complete 3-dimensional subspaces from which a Steiner structure $\dS_2[2,3,13]$ is obtained}
\end{table}

\section{Are there Other Steiner Structures?}
\label{sec:more}

The existence of the Steiner structure $\dS_2[2,3,13]$ provides
some evidence that more Steiner structures exist. But, we certainly
believe that for most  parameters nontrivial Steiner structures do not exist.
We believe that more Steiner structures $\dS_q[2,3,n]$ exist.
\begin{conj}
If $n \equiv 1~(\bmod~6)$, $n \geq 13$, is a prime then there exists
a Steiner structure $\dS_2[2,3,n]$ generated by a set of
$\frac{2^n-2}{42n}$ pairwise disjoint coset complete 3-dimensional
subspaces.
\end{conj}
For $q=3$ we have found that there are no Steiner structure
$\dS_3[2,3,7]$ which has the cyclic shift map
and the Frobenius map as generators of its automorphism groups.
We also found that a Steiner structure $\dS_2[2,4,13]$ with
these automorphisms does not exist.

The Steiner structure $\dS_2[2,3,13]$ that we found is not the
only one with these parameters. By using an invertible
linear transformation, defined by a binary $13 \times 13$
invertible matrix, on the set of 3-dimensional subspaces
of $\dS_2 [2,3,13]$ another Steiner structure with the same
parameters will be obtained. Changing the primitive element
or substituting in a solution another primitive element
instead of $\alpha$ will also produce new Steiner structures
with the same parameters. Some of these structures will have
the cyclic shift map and the Frobenius map as
generators for their automorphism groups. Some of them won't.

We believe that most Steiner structures do not have a nice mathematical
structure and hence it is still probable that a Steiner structure
$\dS_2[2,3,7]$ exists. But, it is almost improbable that it will be
found in the near future.

\section{Other Related Designs}
\label{sec:related}

The existence of a Steiner structure implies the existence of
other combinatorial designs. Moreover, if the steiner structure
is cyclic then more designs are derived.

A \emph{Steiner system} $S(t,k,n)$ is
a set $S$
of $k$-subsets of an $n$-set, say $\Z_n$, such that
each $t$-subset of $\Z_n$ is contained in exactly one
subset of $S$. The following result was obtained in~\cite{EtVa11}.

\begin{theorem}
The existence of a Steiner structure $\dS_2 [2,k,n]$ implies the existence of
a Steiner system $S(3,2^k,2^n)$.
\end{theorem}
\begin{cor}
There exists a Steiner system $S(3,8,8192)$.
\end{cor}

An $(n,w,\lambda)$ \emph{difference family}
is a set of $w$-subsets with elements taken
from an additive group $G$, $|G|=n$, such that each element
of $G \setminus \{ 0 \}$ appears exactly $\lambda$ times
as a difference between elements $w$-subsets.
There is an extensive literature on difference families,
e.g.~\cite{BiEt95,Wil72}. The following theorem
is easily verified.

\begin{theorem}
\label{thm:diffF}
If there exists a Steiner structure $\dS_2 [2,k,n]$ formed
by $\frac{2^n-2}{(2^k-1)(2^k-2)n}$ pairwise disjoint coset complete
$k$-dimensional subspaces,
then there exist a $(2^n-1,2^k-1,1)$ difference family, where the elements
are taken from the group $\Z_{2^n-1}$.
\end{theorem}

Finally, Steiner structures are also \emph{diameter perfect codes} in
the Grassmann scheme~\cite{AAK}. Hence, the new Steiner structure $\dS_2 [2,k,n]$
is a diameter perfect code in the Grassmann scheme.

\section{Conclusion and Future Work}
\label{sec:conclude}

We have presented a framework for a possible
structure of Steiner structures. This framework involves two mappings,
the Frobenius map and the cyclic shift map. These mappings
form automorphism groups in the system. A Steiner structure
$\dS_2[2,3,13]$ was found by computer search based on this frame.
This is the first known nontrivial Steiner structure.
We believe that more Steiner structures exist and for future work we
propose to find more by computer search and to provide a construction for
an infinite family of such structures. One intriguing question is whether
such a structure can be formed from a known construction of a difference family.

\vspace{1cm}

\noindent
{\bf Note added:} This is a preliminary version; an extended version will be
submitted soon. Part of this work appears in an early version~\cite{EtVa12}, which covers all
the material presented on November 1, 2012, in the conference "Trends in Coding Theory",
Ascona, switzerland, October 28, 2012 - November 2, 2012.

\vspace{1cm}

\section*{Acknowledgment}

We started our computer search for a Steiner structure $\dS_2[2,3,13]$
by trying to find a clique~of~size $15$ in the graph on the $105$ cosets
(cf.\ Section\,4).
However, using a backtracking algorithm, we were able to search only
a small fraction of this graph. On October\,30, 2012,
during the conference ``Trends in Coding Theory'' held
in Ascona, Switzerland, Patric \"{O}sterg{\aa}rd suggested
that it would be better to search for an {exact cover}
instead, and pointed out existing algorithms for the exact
cover problem. Following this suggestion, we have changed our programs
and immediately found a~$\dS_2[2,3,13]$ Steiner structure.
We are deeply indebted to Patric \"{O}sterg{\aa}rd; without
his advice, it is not clear whether we would have been able to
complete this work.

Based upon the method presented in Sections 2 and 3 of this paper,
along with \"{O}sterg{\aa}rd's suggestion to look for an exact cover,
Braun and Wasser\-man~\cite{BrWa12} ran a computer search
and independently found numerous instances of $\dS_2[2,3,13]$.



\end{document}